\documentclass[11pt,a4paper]{article}
\usepackage{latexsym}
\usepackage{graphicx}
\usepackage{amsthm}
\usepackage{amsmath}
\usepackage{float}
\usepackage{amssymb}
\newtheorem{Th}{Theorem}

\newcommand{\dd}{{\Delta t}}

\newcommand{\pc}{\bar{\partial}_t}

\newcommand{\ps}{\bar{\partial}_{tt}}
\newcommand{\half}{\frac{1}{2}}

\newcommand{\unp}{{\pmb U}^{n+\frac{1}{2}}}
\newcommand{\unm}{{\pmb U}^{n-\frac{1}{2}}}

\newcommand{\pnp}{P^{n+\frac{1}{2}}}
\newcommand{\pnm}{P^{n-\frac{1}{2}}}

\newcommand{\xnp}{\xi^{n+\frac{1}{2}}}
\newcommand{\xnm}{\xi^{n-\frac{1}{2}}}

\newcommand{\cnp}{{\pmb\chi}^{n+\frac{1}{2}}}

\newcommand{\znp}{\zeta^{n+\frac{1}{2}}}

\newcommand{\fnp}{\phi^{n+\frac{1}{2}}}

\newcommand{\rx}{\xi}
\newcommand{\rc}{{\pmb \chi}}
\newcommand{\re}{{\pmb \eta}}

\newcommand{\ru}{{\pmb u}}
\newcommand{\rv}{{\pmb v}}
\newcommand{\rf}{{\pmb f}}
\newcommand{\rz}{{\pmb z}}

\newcommand{\rU}{{\pmb U}}
\newcommand{\rR}{{\pmb R}}
\newcommand{\rH}{{\pmb H}}
\newcommand{\rV}{{\pmb V}}

\newcommand{\la}{\lambda}

\newcommand{\pn}{P^{n;\theta}}
\newcommand{\nt}{{n;\theta}}

\newcommand{\nd}{\nabla\cdot}

\title{A Priori Error Estimates for Mixed Finite Element
$\theta$-Schemes for the Wave Equation\thanks{This research was supported by
 Sultan Qaboos University under Grant
IG/SCI/DOMS/13/02.}}

\author{Samir Karaa\thanks{Department of Mathematics and Statistics, Sultan
Qaboos University, P. O. Box 36, Al-Khod 123, Muscat,
Sultanate of Oman ({\tt skaraa@squ.edu.om}).}}

\date{}

\begin{document}

\maketitle

\begin{abstract}

A family of implicit-in-time mixed finite element schemes is presented
for the  numerical approximation of the acoustic wave equation.
The mixed space discretization is based on the displacement form of the 
wave equation  and the time-stepping method employs a three-level 
one-parameter scheme.
A rigorous stability analysis is presented based on  
energy estimation and sharp stability results are obtained. 
A convergence analysis is carried out and
optimal a priori $L^\infty(L^2)$ error estimates for both displacement and 
pressure are derived.

\end{abstract}


{\bf Mathematical subject codes} 65L05, 65M12, 65M60, 65M15

\pagestyle{myheadings}
\thispagestyle{plain}

\section{Introduction}

The acoustic wave equation is used to model the effects of wave propagation 
in heterogeneous media. 
Solving this equation efficiently  is of fundamental importance in 
many real-life problems.  In geophysics,
it helps for instance in the interpretation of the seismograph field data 
and to predict damage patterns due to earthquakes. 
Using finite element methods for its approximation 
is attractive because of the ability to handle complex discretizations and design adaptive grid refinement 
strategies  based on error indicators. 

Previous attempts on wave simulation by finite elements have used
continuous Galerkin methods \cite{Baker-3,Bao,CJRT,Dupont-15,Marfurt,Rauch},
mixed finite element methods \cite{CDW-12,CDW-13,Geveci-19,JRW,Pani-2001,Pani-2004,VT-2008}, and
discontinuous Galerkin  methods \cite{GS-2009,Johnson,RW-27,RW-28}.
In a mixed finite element formulation both displacements and stresses are
approximated simultaneously. 
This approach provides higher-order approximations to the
stresses. 
This property is important in many problems, in particular 
in modeling boundary controlability of the wave equation \cite{GKW-89}. 
One of the main difficulties of the mixed finite element techniques is the 
requirement of compatibility of the approximating spaces for convergence 
and stability.

Given a bounded convex polygonal domain $\Omega$ in $\mathbb{R}^m$, $m=2,\,3$, with boundary
$\partial\Omega=\Gamma_D\cup\Gamma_N$, and unit outward normal  
${\pmb \nu}$, the general form of the wave equation is
\begin{eqnarray}
\rho\ru_{tt} +\nd {\pmb{\tilde \tau}}&=&\rf \qquad \mbox{ in } \Omega\times
(0,T),\label{eq:o1}\\
\nd\ru &=&0 \qquad \mbox{ on } \Gamma_D\times (0,T),\label{eq:o2}\\
\ru\cdot{\pmb \nu} &=&0 \qquad \mbox{ on } \Gamma_N\times (0,T),\label{eq:o3}\\
\ru(\cdot,0)&=&\ru^0 \qquad \mbox{ in } \Omega,\label{eq:o4}\\
\ru_t(\cdot,0)&=&\rv^0 \qquad \mbox{ in } \Omega,\label{eq:o5}
\end{eqnarray}
where $\ru$ is the displacement, $\rho$ is the density, and ${\pmb{\tilde\tau}}$ 
is the stress tensor given by the generalized Hooke's law 
${\pmb{\tilde\tau}}=\lambda(\nd\ru){\pmb{\tilde I}}+
\mu(\nabla\ru+(\nabla\ru)^T)$. Here $\lambda>0$ and $\mu$  are the Lam\'e 
coefficients characterizing the material. The function $\rf$ represents a 
general source term and $\ru^0$ and $\rv^0$ are initial conditions on 
displacements and velocities. We assume that $\rf$, $\ru^0$ and $\rv^0$
are smooth enough so that there is a unique solution   
$\ru\in{\cal C}^2((0,T)\times \Omega)$ to (\ref{eq:o1})-(\ref{eq:o5}), see
\cite{Knops-Payne-1971}.
 
The limiting case of (\ref{eq:o1}) with $\mu=0$ is referred to as the acoustic 
wave equation, which is
\begin{equation}
\label{eq:wave1}
\rho\ru_{tt} +\nd(\lambda(\nd\ru){\pmb{\tilde I}})=\rf.
\end{equation}
It is assumed that $\rho$ and $\lambda$ are bounded below and  above by the positive 
constants $\rho_0$, $\rho_1$, $\lambda_0$, and $\lambda_1$, respectively. 
This vector equation is equivalent to the scalar wave equation after making 
the substitution $p=\lambda\nd·u$. 
The mixed method is established by using this relationship, leading to the 
coupled system
\begin{eqnarray}
\rho\ru_{tt} -\nabla p&=&\rf \qquad \qquad \mbox{ in } \Omega\times
(0,T),\label{eq:oo1}\\
\lambda^{-1}p&=&\nd\ru \qquad \mbox{ in } \Omega\times(0,T),\label{eq:oo2}
\end{eqnarray}
with the appropriate boundary and initial conditions.

A priori error estimates for solving (\ref{eq:oo1})-(\ref{eq:oo2}) were obtained
in \cite{CDW-12,CDW-13,Geveci-19,JRW}. In \cite{Geveci-19}, Geveci
derived $L^\infty$-in-time, $L^2$-in-space error bounds for the
continuous-in-time mixed finite element approximations of velocity and stress.
In \cite{CDW-12,CDW-13}, a priori error estimates were obtained for the mixed
finite element approximation of displacement which  requires less regularity 
than was needed in \cite{Geveci-19}. Stability for a family of discrete-in-time
schemes was also demontratred. In \cite{JRW}, an
alternative mixed finite element displacement formulation was proposed 
reducing requirement on the regularity on the displacement variable. 
For the explicit discrete-in-time problem,
stability results were established and error estimates were obtained.
The effectiveness of the method analyzed in \cite{JRW} was demonstrated in 
\cite{Jenkins}  by  providing simulations using both lowest-order and 
next-to-lowest-order Raviart–-Thomas elements on rectangles \cite{R16}. 


The purpose of this paper is to analyze an implicit time-stepping method 
combined with the mixed finite
element discretization proposed in \cite{JRW}. We prove the stability of the
proposed method by using energy estimation, and show in particular that it conserves 
certain energy. 
We also invertigate the convergence of the method
and prove optimal a priori $L^\infty(L^2)$ error estimates 
for both displacement and pressure. The rest of the paper is organized as
follows. In sections 2 and 3, we introduce notations and
describe the weak formulation of the problem. The fully discrete 
mixed finite element method is presented in section 4. 
Stability results are established in section 4 and optimal a priori error 
estimates are obtained in section 5. 
Conclusions are given in the last section.
 


\section{Notation}
\label{sect:Notation}
\setcounter{equation}{0}
We shall use the following inner products and norms in this paper. The $L^2$-inner
product over $\Omega$ is defined by
$$
(u,v)=\int_\Omega uv\,d\Omega,
$$
inducing the $L^2$-norm over $\Omega$, $||v||_{L^2(\Omega)}=(v,v)^{1/2}$. 
The inner product over the boundary $\partial \Omega$ is denoted by
$$
\langle u,v\rangle=\int_{\partial \Omega} uv\,d\Omega
$$
for $u$, $v\in H^{\half+\varepsilon}(\Omega)$ with $\varepsilon>0$.
We introduce the time-space norm:
$$
||v||_{L^2(0,T;L^2(\Omega))}=||v||_{L^2(L^2)}=\left(\int_0^T||v||^2_{L^2(\Omega)}dt
\right)^{\half}.
$$
The time-space norm  $||\cdot||_{L^\infty(L^2)}$ is similarly defined.
In addition to the $L^2$ spaces, we use the standard Sobolev space for mixed
methods:
$$
\rH(\Omega,\mbox{div})=\{\rv:\rv\in(L^2(\Omega))^m,\nd \rv\in L^2(\Omega)\},
$$
with associated norm
$$
||\rv||_{\rH(\Omega,\mbox{div})}=  ||\rv||_{L^2(\Omega)}+||\nd \rv||_{L^2(\Omega)},
$$
where
$$
||\rv||_{L^2(\Omega)}=\left(\sum_{i=1}^m||v_i||_{L^2(\Omega)}^2\right)^\half.
$$

For the time discretization, we adopt the following notation. Let
$N$ be a positive integer, $\dd=T/N$, and $t^n=n\dd$. For any function
$v$ of time, let $v^n$ denote $v(t^n)$. We shall use this notation for functions
defined for all times as well as those defined only at discrete times. Set
$$
v^{n+\half}=\half\left(v^{n+1}+v^n\right),
$$
$$
\pc v^{n+\half}=\frac{1}{\dd}\left(v^{n+1}-v^n\right),
$$
$$
\pc v^{n}=\frac{1}{2\dd}\left(v^{n+1}-v^{n-1}\right),
$$
$$
\ps v^{n}=\frac{1}{\dd^2}\left(v^{n+1}-2v^n+v^{n-1}\right),
$$
$$
v^{n;\theta}=\theta v^{n+1}+(1-2\theta)v^n+\theta v^{n-1},
$$
where $0\leq\theta\leq 1$. We also define the discrete $l^\infty$-norm
for  time-discrete functions by
$$
||v||_{l^\infty_\dd(0,T;L^2(\Omega))}=||v||_{l^\infty(L^2)}=
\max_{0\leq n\leq N}||v^n||_{L^2(\Omega)}.
$$
%


\section{Weak Formulation}
\setcounter{equation}{0}


The finite element approximation of the wave problem is based
on its weak formulation which is derived in the usual 
manner. Integrating by parts and using the data on the boundary  of
$\Omega$, we obtain the weak formulation \cite{JRW}: For any $t\geq 0$, find 
$(\ru(t),p(t))\in \rV\times W$ such that
%
\begin{eqnarray}
(\ru(0),\rv)&=&(\ru^0,\rv)\qquad \forall \rv\in \rV,\label{eq:w1-a}\\
(\ru_t(0),\rv)&=&(\rv^0,\rv)\qquad \forall \rv\in \rV,\label{eq:w1-b}\\
(\lambda^{-1}p(0),w)&=&(\nabla\cdot \ru^0,w)\qquad \forall w\in W,\label{eq:w1-c}\\
(\rho\ru_{tt}(t) ,\rv)+(p(t),\nd \rv)&=&(\rf(t),\rv)\qquad \forall \rv\in \rV,
\quad t>0,\label{eq:w1}\\
(\lambda^{-1}p(t),w)-(\nd \ru(t),w)&=&0\qquad \forall w\in W,\quad t>0, \label{eq:w2}
\end{eqnarray}
where $\rV$ and $W$ are given by
$$
\rV=\{\rv\in \rH(\Omega,\mbox{div})\,:\,\rv\cdot{\pmb\nu}|_{\Gamma_N}=0\},
$$
$$
W=H^{\half+\varepsilon}(\Omega) \mbox{ for any } \varepsilon>0.
$$
The present formulation requires less regularity on the displacement than
standard approaches. For instance in \cite{CDW-12,CDW-13} it is
necessary that $\nabla p\in \rH(\Omega,\mbox{div})$ so that 
$\nd\ru\in H^2(\Omega)$. Here, it is only required that $\nd\ru\in H^\half$, 
and it can be verified that the solution $\ru$ of problem 
(\ref{eq:o1})-(\ref{eq:o5}) with $p=\lambda\nd\ru$ is a solution to 
(\ref{eq:w1})-(\ref{eq:w2}), see \cite{JRW}.

Differentiate (\ref{eq:w2}) with respect to time to obtain
\begin{equation}\label{eq:w3}
(\lambda^{-1}p_t,w)-(\nd \ru_t,w)=0\qquad \forall w\in W.
\end{equation}
We next assume $\rf=0$ and choose $\rv=\ru_t$ and $w=p$ in (\ref{eq:w1}) and (\ref{eq:w3}), respectively, so that
\begin{eqnarray}
(\rho\ru_{tt} ,\ru_t)+(p,\nd \ru_t)&=&0,\label{eq:w4}\\
(\lambda^{-1}p_t,p)-(\nd \ru_t,p)&=&0 \label{eq:w5}.
\end{eqnarray}
By adding the two equations, we find that
\begin{equation}\label{eq:w6}
(\rho\ru_{tt} ,\ru_t)+(\lambda^{-1}p_t,p)=0,
\end{equation}
or
$$
\half \frac{d}{dt}\left|\left|\rho^\half\ru_{t} \right|\right|^2_{L^2(\Omega)}
+\half \frac{d}{dt}\left|\left|\la^{-\half} p \right|\right|^2_{L^2(\Omega)}=0.
$$
Thus, in the absence of forcing, the (continuous) energy
\begin{equation}\label{eq:w7}
\half \left|\left|\rho^\half\ru_{t} \right|\right|^2_{L^2(\Omega)}
+\half \left|\left|\la^{-\half} p \right|\right|^2_{L^2(\Omega)}
\end{equation}
is conserved for all time. It will be shown that a similar form of energy is conserved 
by the numerical solution of the wave problem.

\section{Finite Element Approximation}
\setcounter{equation}{0}
For the finite element approximation, we
let $\{{\cal E}_h\}_{h>0}$ be a quasi-uniform family of finite element 
partitions of $\Omega$, where $h$ is the maximum element diameter. 
Let $\rV_h\times W_h$ be any of the usual mixed finite element approximating 
subspaces of $\rV\times W$, that is, the Raviart-Thomas-Nedelec spaces 
\cite{R15,R16}, Brezzi-Douglas-Marini spaces \cite{R5}, or 
Brezzi-Douglas-Fortin-Marini spaces \cite{R4}.
For each of these mixed spaces there is a projection 
$\Pi_h:\rH(\Omega,\mbox{div})\rightarrow \rV_h$ such that
for any $\rz\in \rH(\Omega,\mbox{div})$
\begin{equation}\label{eq:w8}
(\nd\Pi_h\rz,w)=(\nd\rz,w)\quad \forall w\in W_h.
\end{equation}
We have the property that, if $\rz\in \rH(\Omega,\mbox{div})\cap \rH^k(\Omega)$, 
then 
\begin{equation}\label{eq:w9}
||\Pi_h\rz-\rz||_0\leq C h^j||\rz||_j, \quad 1\leq j\leq k,
\end{equation}
where $k$ is associated with the degree of polynomial and $||\cdot||_s$ is 
the standard Sobolev norm on $(H^s(\Omega))^m$. Here and in what 
follows, $C$ is a generic positive constant which is independent of $h$ and $\dd$.

For $\phi\in W$, we denote by ${\cal P}_h\phi$ the $L^2$-projection of $\phi$ onto $W_h$
defined by requiring that
\begin{equation}\label{eq:w10}
({\cal P}_h\phi,w)=(\phi,w)\quad\forall w\in W_h.
\end{equation}
If $\phi\in W\cap H^k(\Omega)$, then we also have
\begin{equation}\label{eq:w11}
||{\cal P}_h\phi-\phi||_s\leq C h^{j-s}||\phi||_j, \quad 0\leq s\leq k,\quad 0\leq j\leq k.
\end{equation}
The semidiscrete mixed finite element approximation to $(\ru(t),p(t))$
is to seek $(\rU(t),P(t))\in \rV_h\times W_h$ satisfying
\begin{eqnarray}
(\rU(0),\rv)&=&(\Pi_h\ru^0,\rv)\quad \forall \rv\in \rV_h,\label{eq:ww1}\\
(\rU_t(0),\rv)&=&(\Pi_h\rv^0,\rv)\quad \forall \rv\in \rV_h,\label{eq:ww2}\\
(P(0),w)&=&(p(0),w)\quad \forall w\in W_h,\label{eq:ww3}\\
(\rho\rU_{tt}(t),\rv)+(P(t),\nd\rv)  &=&(\rf(t),\rv)\quad \forall \rv\in
\rV_h,\quad t>0,\label{eq:ww4}\\
(\lambda^{-1}P(t),w)-(\nd\rU(t),w)&=&0\qquad\qquad  \forall w\in W_h,\quad
t>0.\label{eq:ww5}
\end{eqnarray}
Existence and uniqueness of a solution $(\rU(t),P(t))$ to 
the variational problem (\ref{eq:ww1})-(\ref{eq:ww5}) is shown in \cite{JRW}.


The fully discrete mixed finite element $\theta$-scheme is then defined by 
finding a sequence of pairs $(\rU^{n},P^{n})\in \rV_h\times W_h$, 
$0\leq n\leq N$, such that 
\begin{eqnarray}
(\rU^0,\rv)&=&(\Pi_h \ru^0,\rv)\quad \forall \rv\in \rV_h,\label{eq:www1}\\
(P^0,w)&=&(p^0,w)\quad \forall w\in W_h,\label{eq:www2}\\
\left(\rho\pc\rU^\half ,\rv\right)+\theta^2\dd\left(\pc
P^\half,\nd\rv\right)
+\frac{\dd}{2}(P^0,\nd\rv)&=&\left(\frac{\dd}{2}\rf^0+\theta\dd^2\pc\rf^\half,\rv\right)\nonumber\\
&& +\left(\rho\Pi_h\rv^0,\rv\right)\quad \forall \rv\in
\rV_h, \label{eq:www3}\\
(\rho\ps\rU^n,\rv)+(P^{n;\theta},\nd\rv)  &=&(\rf^{n;\theta},\rv)\quad \forall
\rv\in \rV_h,\label{eq:www4}\\
(\lambda^{-1}P^{n+1/2},w)-(\nd\rU^{n+1/2},w)&=&0\quad \forall w\in
W_h.\label{eq:www5}
\end{eqnarray}
Equation (\ref{eq:www3}) is derived from the following expansion:
$$
\ru^{1}=\ru^0+\dd \rv^0+\dd^2\left[\theta
\ru^{1}_{tt}+\left(\frac{1}{2}-\theta\right)\ru^{0}_{tt}\right]+{\cal O}(\dd^3).
$$
The present $\theta$-scheme is explicit in time if $\theta=0$ and implicit otherwise. 
The existence and uniqueness of a solution to the resulting linear system 
for a nonzero value of $\theta$ 
follows from the unisolvancy of the mixed formulation of the following elliptic
problem:
\begin{eqnarray*}
\nd (\lambda\nabla \phi)+\frac{1}{\theta\dd^2}\rho \phi &=&0 \qquad \mbox{ in } \Omega,\\
\phi &=&0 \qquad \mbox{ on } \partial\Omega.
\end{eqnarray*}
The explicit case has been considered in \cite{JRW}.
As expected from an explicit scheme, the method is conditionally stable.
As a stability constraint, it requires to choose
$$\dd={\cal O}(h).$$
In the next sections, stability and convergence properties of the  
proposed $\theta$-scheme are analyzed.



\section{Stability Analysis}
\setcounter{equation}{0}
We derive sharp stability bounds based on the energy technique 
and show that the proposed scheme conserves certain energy. 
We consider (\ref{eq:www4}) and (\ref{eq:www5}) for the homogeneous case
\begin{eqnarray}
(\rho\ps \rU^n,\rv)+(\pn,\nd \rv)&=&0\qquad \forall \rv\in \rV_h,\label{eq:a1}\\
(\lambda^{-1}P^{n+1/2},w)-(\nd \rU^{n+1/2},w)&=&0\qquad \forall w\in W_h \label{eq:a2}.
\end{eqnarray}
%
We will make use of the {\it inverse assumption}, 
which states that there exists a constant
$C_0$ independent of $h$, such that
\begin{equation}\label{eq:ii}
||\nd\phi||_{L^2(\Omega)}\leq C_0 h^{-1}||\phi||_{L^2(\Omega)}
\end{equation}
for all $\phi\in W_h$.
The following stability result holds.
\begin{Th}\label{th:1}
The fully discrete scheme $(\ref{eq:www1})$-$(\ref{eq:www5})$  is stable if
\begin{equation}
\label{eq:cfl-1}
\dd^2\left(\frac{1}{4}-\theta\right)\frac{C_0^2\lambda_1}{h^2\rho_0}\leq 1,
\end{equation}
and conserves the discrete energy
\begin{equation}
\label{eq:cfl-ee}
E_h^{n+\half}=\frac{1}{2}\left[||\rho^{\half}\pc \rU^{n+\half}||^2+
\dd^2\left(\theta-\frac{1}{4}\right)
||\lambda^{-\half}\pc P^{n+\half}||^2+||\lambda^{-\half}P^{n+\half}||^2\right].
\end{equation}
The scheme is unconditionally stable if  $\theta\geq 1/4$.
\end{Th}
\begin{proof}  If we subtract (\ref{eq:a2}) from itself, with $n+1/2$ replaced 
by $n-1/2$, we find that
\begin{equation}
\label{eq:s1}
(\lambda^{-1}(P^{n+1}-P^{n-1},w)-(\nabla\cdot(U^{n+1}-U^{n-1}),w)=0.
\end{equation}
As (\ref{eq:a1}) holds for all $\rv\in \rV_h$ and (\ref{eq:s1}) holds for all 
$w\in W_h$,
we choose $\rv=\pc \rU^n$ and $w=\frac{\pn}{2\dd}$ so that
\begin{eqnarray}
(\rho\ps \rU^n,\pc \rU^n)+(\pn,\nd \pc \rU^n)&=&0,\label{eq:s2}\\
(\lambda^{-1}\pc P^n,\pn)- (\nd \pc \rU^n,\pn)&=&0\label{eq:s3}.
\end{eqnarray}
By adding (\ref{eq:s2}) and (\ref{eq:s3}) we obtain
\begin{equation}
\label{eq:s4}
(\rho\ps \rU^n,\pc \rU^n)+(\lambda^{-1}\pc P^n,\pn)=0.
\end{equation}
Note that
\begin{eqnarray}
\label{eq:s5}
\pn&=&\dd^2\theta\ps P^n+P^n\nonumber\\
&=&\dd^2\left(\theta-\frac{1}{4}\right)\ps
P^n+\frac{1}{2}\left(P^{n+\half}+P^{n-\half}\right).
\end{eqnarray}
%
%
Hence, (\ref{eq:s4}) can be rewritten as
\begin{equation}
\label{eq:s6}
(\rho\ps \rU^n,\pc \rU^n)+
\dd^2\left(\theta-\frac{1}{4}\right)(\lambda^{-1}\ps P^n,\pc P^n)+
\frac{1}{2}\left(\lambda^{-1}(P^{n+\half}+P^{n-\half}),\pc P^n\right)=0.
\end{equation}
Using that
$$
\bar\partial_{t}\rU^n=\frac{\pc\rU^{n+\half}+\pc\rU^{n-\half} }{2},\qquad
\ps \rU^n=\frac{\pc\rU^{n+\half}-\pc\rU^{n-\half}}{\dd},
$$
we deduce that
\begin{eqnarray*}
(\rho\ps \rU^n,\pc \rU^n) &=&
\frac{1}{2\dd}\,(\rho\pc \unp -\rho\pc \unm,\pc\unp+\pc\unm)\\
&=&
\frac{1}{2\dd}\left[ (\rho\pc \unp,\pc \unp)- (\rho\pc \unm,\pc \unm)  \right],
\end{eqnarray*}
and similarly
$$
(\lambda^{-1}\ps P^n,\pc P^n) = \frac{1}{2\dd}
\left[ (\lambda^{-1}\pc \pnp,\pc \pnp)- (\lambda^{-1}\pc \pnm,\pc \pnm)  \right].
$$
We also have
\begin{eqnarray*}
\left(\lambda^{-1}(\pnp+\pnm),\pc P^n\right) &=&
\frac{1}{\dd}\,(\lambda^{-1}\pnp+\lambda^{-1}\pnm,\pnp-\pnm)\\
&=&
\frac{1}{\dd}\left[ (\lambda^{-1}\pnp,\pnp)-(\lambda^{-1}\pnm,\pnm) \right].
\end{eqnarray*}
Hence,  (\ref{eq:s6}) is equivalent to
$$
\frac{1}{\dd}\left(E_h^{n+\half}-E_h^{n-\half}\right)=0,
$$
where $E^{n+\half}_h$ is the quantity defined by (\ref{eq:cfl-ee}).
This relation indicates that $E_h^{n+\half}$ is conserved for all
time,  which guarantees the stability of the scheme
if and only if $E_h^{n+\half}$ defines a positive energy. A sufficient condition 
is that
$$
\left|\left|\rho^{\half}\pc\unp\right|\right|^2+\dd^2\left(\theta-\frac{1}{4}\right)
\left|\left|\lambda^{-\half}\pc\pnp\right|\right|^2\geq 0
$$
for all $n\geq 0$.  
Clearly, the scheme is unconditionally stable when $\theta\geq 1/4$. Now, 
using Cauchy-Schwarz inequality and the inverse assumption (\ref{eq:ii}), we
obtain
\begin{eqnarray*}
\left(\lambda^{-1}\pc\pnp,w\right)&=&\left(\nd\pc\unp,w\right)\\
&\leq& \left|\left|\nd\pc\unp\right|\right|_{L^2(\Omega)}||w||_{L^2(\Omega)}\\
&\leq&
\frac{C_0}{h}\left|\left|\pc\unp\right|\right|_{L^2(\Omega)}||w||_{L^2(\Omega)}.
\end{eqnarray*}
By setting $w=\pc\pnp$, we see that
\begin{eqnarray*}
\left|\left|\lambda^{-\half}\pc \pnp\right|\right|^2_{L^2(\Omega)}
&\leq& \frac{C_0}{h}\left|\left|\pc\unp\right|\right|_{L^2(\Omega)}
\left|\left|\pc \pnp\right|\right|_{L^2(\Omega)}\\
&\leq& \frac{C_0 \lambda_1^{\half}}{h\rho_0^\half}\left|\left|\rho^\half\pc\unp\right|\right|_{L^2(\Omega)}
\left|\left|\lambda^{-\half}\pc \pnp\right|\right|_{L^2(\Omega)},
\end{eqnarray*}
or
$$
\left|\left|\lambda^{-\half}\pc \pnp\right|\right|_{L^2(\Omega)}\leq
\frac{C_0
\lambda_1^{\half}}{h\rho_0^\half}\left|\left|\rho^{\half}\pc\unp\right|\right|_{L^2(\Omega)}.
$$
Hence,  a sufficient condition for stability is given by
$$
||\rho^{\half}\pc\unp||^2+\dd^2\left(\theta-\frac{1}{4}\right)
\frac{C_0^2\la_1}{h^2\rho_0}||\rho^{\half}\pc\unp||^2\geq 0,
$$
which completes the proof.
\end{proof}

 
The case with $\theta=1/4$  is interesting
because the form of the discrete energy in this case
is similar to that of the continuous problem. In addition, one can verify that 
the time truncation error is minimized over the set of all $\theta\geq 1/4$  
when $\theta=1/4$.


\section{Convergence Analysis}
\setcounter{equation}{0}

In this section, we prove optimal convergence of the fully discrete finite 
element solution  in the $L^\infty(L^2)$ norm.
Some of the techniques used in the
proofs can be found in previous works  \cite{Kar-FE-Theta, Karaa-2012}.
In order to estimate the errors in the finite element approximation,
we define the auxiliary functions
$$
\rc^n=\rU^n-\Pi_h\ru^n,\quad \re^n=\ru^n-\Pi_h\ru^n, \qquad \rx^n=P^n-{\cal P}_hp^n,
\qquad \zeta=p^n-{\cal P}_hp^n,
$$
where $\Pi_h$ and ${\cal P}_h$ are defined in Section~4.
From (\ref{eq:w1})-(\ref{eq:w2}) and (\ref{eq:www4})-(\ref{eq:www5}),
and the properties of the projections
$\Pi_h$ and ${\cal P}_h$, we arrive at
\begin{eqnarray}
(\rho\ps \rc^n,v)+(\rx^\nt,\nd \rv)&=& (\rho\ps \re^n,\rv)+({\pmb r}^n,\rv)\qquad
\forall \rv\in \rV_h,\quad n\geq 1,\label{eq:b1}\\
(\lambda^{-1}\rx^{n+1/2},w)-(\nd \rc^{n+1/2},w)&=&(\lambda^{-1}\zeta^{n+1/2},w)\quad
\qquad\forall w\in W_h,\quad n\geq 0,\label{eq:b2}
\end{eqnarray}
where ${\pmb r}^n=\rho (\ru^\nt_{tt}-\ps \ru^n)$. 
Another equation has to be derived for the initial errors $\rc^1$ and $\rx^1$.  
Consider (\ref{eq:w1}) at $n=0$ and $n=1$, respectivey,  
and subtract the resulting equations so that
\begin{equation}\label{eq:cv2}
\left(\rho\pc\ru_{tt}^\half,\rv\right)+\left(\pc p^\half,\nd\rv\right)=
\left(\pc \rf^\half,\rv\right).
\end{equation}
A use of Taylor's formula with integral remainder yields
\begin{equation}\label{eq:cv1}
\pc\ru^\half=\rv^0+\frac{\dd}{2}\ru_{tt}^0+\frac{1}{2\dd}
\int_0^\dd (\dd-t)^2\frac{\partial^3 \ru}{\partial t^3}(t)\,dt.
\end{equation}
Using (\ref{eq:cv2}) and (\ref{eq:cv1}), we readily obtain
\begin{eqnarray}\label{eq:cv3}
\left(\rho\pc\ru^\half,\rv\right)+\theta\dd^2\left(\pc p^\half,\nd\rv\right)&=&
\theta\dd^2\left(\pc\rf^\half,\rv\right)-\theta\dd^2
\left(\rho\pc\ru_{tt}^\half,\rv\right)\nonumber\\
&&+(\rho\rv^0,\rv)+\frac{\dd}{2}(\rho\ru_{tt}^0,\rv)\nonumber\\
&&+\frac{1}{2\dd}
\int_0^\dd (\dd-t)^2\left(\rho\frac{\partial^3 \ru}{\partial t^3},\rv\right)\,dt.
\end{eqnarray}
%
%
%
Subtracting (\ref{eq:cv3}) from (\ref{eq:www3}) and taking into account
(\ref{eq:www2}) and (\ref{eq:w1}) to arrive that
\begin{equation}\label{eq:cv5}
\begin{split}
\left(\rho\pc\rc^\half ,\rv\right)+\theta\dd^2\left(\pc \rx^\half,\nd\rv\right)
&+\frac{\dd}{2}\left(\rx^0,\nd\rv\right)=
\left(\rho\pc\re^\half
,\rv\right)+\left(\rho(\Pi_h\rv^0-\rv^0),\rv\right)\\
&
+\theta\dd^2\left(\rho\pc \ru_{tt}^\half,\rv\right)
-\frac{1}{2\dd}
\int_0^\dd (\dd-t)^2\left(\rho\frac{\partial^3 \ru}{\partial t^3},\rv\right)\,dt.
\end{split}
\end{equation}
%
Note that $\rx^0=0$ and $\rc^0=0$. We now state and prove our convergence result.
\begin{Th}\label{Th:1} If $\ru\in L^\infty({\pmb H}(\Omega;{\rm div}))$,
$\frac{\partial^3 \ru}{\partial t^3}\in L^1({\pmb L}^2(\Omega))$,
$\frac{\partial^4 \ru}{\partial t^4}\in L^\infty({\pmb L}^2(\Omega))$,
and $p\in L^\infty(L^2(\Omega))$,
then for $\{\rU^n,P^n\}$ defined by
$(\ref{eq:www1})$-$(\ref{eq:www5})$ there exists a constant $C$ independent of
$h$ and $\dd$ such that if
\begin{equation}
\label{eq:cfl2}
\dd^2\left(\frac{1}{4}-\theta\right)\frac{\lambda_1C_0^2}{\rho_0h^2}<\frac{1}{2},
\end{equation}
then the following a priori error estimate holds:
\begin{equation}\label{eq:cv7}
\begin{split}
\left|\left|\rho^\half(\ru-\rU)\right|\right|_{l^\infty(L^2)}+&
\left|\left|\la^{-\half}(p-P)\right|\right|_{l^\infty(L^2)}\\
&
 \leq C(h^r+\dd^2)\left(||\ru||_{L^\infty(H^r)}+
\left|\left|\frac{\partial^3 \ru}{\partial t^3} \right|\right|_{L^\infty(L^2)}+
 ||p||_{L^\infty(L^2)}\right),
\end{split}
\end{equation}
where $r$ is associated with the degree of the finite element polynomial.
\end{Th}
\begin{proof} We first rearrange (\ref{eq:b1}) in the form
\begin{equation}
\label{eq:b3}
(\rho\ps \rc^n,\rv)+\dd^2\left(\theta-\frac{1}{4}\right) (\ps \rx^n,\nd \rv)
+\half\left(\xnp+\xnm,\nd \rv\right)
= (\rho\ps \re^n,v)+({\pmb r}^n,\rv).
\end{equation}
Summing over time levels and multiplying through by $\dd$ yields
\begin{equation}
\label{eq:b4}
\begin{split}
(\rho\pc \cnp-\rho\pc &\rc^{\half},\rv)+\dd^2\left(\theta-\frac{1}{4}\right)
\left(\pc \xnp-\pc\rx^\half,\nd \rv\right)\\
&+\frac{\dd}{2}\sum_{i=1}^{n}\left(\rx^{i+\half}+\rx^{i-\half},\nd \rv\right)
=\left(\rho\pc \re^{n+\half}-\rho\pc
\re^{\half},\rv\right)+\left(\dd\sum_{i=1}^{n}{\pmb r}^i,v\right).
\end{split}
\end{equation}
Upon defining
$$
\phi^0=0, \qquad \phi^n=\dd\sum_{i=0}^{n-1}\rx^{i+\half},
$$
we verify that
$$
\phi^{n+\half}=\frac{\dd}{2} \rx^\half+\frac{\dd}{2}
\sum_{i=1}^{n}\left(\rx^{i+\half} +\rx^{i-\half}\right).
$$
Taking into account (\ref{eq:cv5}) and that
$\pc\rx^\half=\frac{2}{\dd}\rx^\half$,  (\ref{eq:b4}) becomes
\begin{equation}
\label{eq:b5}
(\rho\pc \cnp,\rv)+\dd^2\left(\theta-\frac{1}{4}\right)
\left(\pc \xnp,\nd \rv\right)+
\left(\phi^{n+\half},\nd \rv\right)=\left(\rho\pc\re^{n+\half},\rv\right)
+\left(\rR^n,\rv\right),
\end{equation}
where
\begin{eqnarray*}
\rR^n=\dd\sum_{i=1}^{n}{\pmb r}^i+\rho(\Pi_h\rv^0-\rv^0)+\theta\dd^2
\rho\pc \ru_{tt}^\half
-\frac{1}{2\dd}
\int_0^\dd (\dd-t)^2\rho\frac{\partial^3 \ru}{\partial t^3}\,dt.
\end{eqnarray*}
Since $\pc\phi^{n+\half}=\xnp$, (\ref{eq:b2}) reads
\begin{equation}
\label{eq:b6}
\left(\lambda^{-1}\pc\phi^{n+\half},w\right)-\left(\nd \cnp,w\right)=
\left(\lambda^{-1}\znp,w\right).
\end{equation}
Choosing $v=\cnp$ and $w=\fnp$ in (\ref{eq:b5}) and (\ref{eq:b6}),
respectively, and adding the resulting equations, we arrive at
\begin{equation}
\label{eq:b7}
\begin{split}
(\rho\pc \cnp,\cnp)+&\dd^2\left(\theta-\frac{1}{4}\right)
\left(\pc \xnp,\nd\cnp \right)+\left(\lambda^{-1}\pc\fnp,\fnp\right)\\
&
=\left(\rho\pc\re^{n+\half},\cnp\right)+\left(\rR^n,\cnp\right)+\left(\lambda^{-1}\znp,\fnp\right).
\end{split}
\end{equation}
Again, choose $w=\pc\xnp$ in (\ref{eq:b2}) so that
$$
\left(\pc\xnp,\nd\cnp\right)=\left(\lambda^{-1}\xnp,\pc\xnp\right)-\left(\lambda^{-1}\znp,
\pc\xnp\right).
$$
Substitution into (\ref{eq:b7}) yields
\begin{equation}
\label{eq:b8}
\begin{split}
(\rho\pc \cnp,\cnp)+&\dd^2\left(\theta-\frac{1}{4}\right)
\left(\lambda^{-1}\pc\xnp,\xnp\right)
+\left(\lambda^{-1}\pc\fnp,\fnp\right)\\
=&\dd^2\left(\theta-\frac{1}{4}\right) \left(\lambda^{-1}\znp,\pc\xnp\right)+
\left(\rho\pc\re^{n+\half},\cnp\right)\\
&+\left(\rR^n,\cnp\right)+\left(\lambda^{-1}\znp,\fnp\right).
\end{split}
\end{equation}
The terms on the right-hand side of (\ref{eq:b8}) are bounded using Cauchy-Schwarz
inequality as
\begin{eqnarray*}
\left(\lambda^{-1}\znp,\pc\xnp\right)&\leq&
\left|\left|\lambda^{-\half}\znp\right|\right|_{L^2(\Omega)}
\left|\left|\lambda^{-\half}\pc\xnp\right|\right|_{L^2(\Omega)}\\
\left(\rho\pc\re^{n+\half},\cnp\right)&\leq&
\left|\left|\rho \pc\re^{n+\half}  \right|\right|_{L^2(\Omega)}
\left|\left| \cnp  \right|\right|_{L^2(\Omega)}\\
\left(\rR^n,\cnp\right)&\leq&
\left|\left| \rR^n  \right|\right|_{L^2(\Omega)}
\left|\left| \cnp  \right|\right|_{L^2(\Omega)}\\
\left(\lambda^{-1}\znp,\fnp\right)&\leq&
\left|\left| \lambda^{-\half}\znp \right|\right|_{L^2(\Omega)}
\left|\left| \lambda^{-\half}\fnp \right|\right|_{L^2(\Omega)}.
\end{eqnarray*}
We now distinguish the cases where $\theta\geq \frac{1}{4}$ and
$\theta< \frac{1}{4}$.
In the first case, we
sum on (\ref{eq:b8}) over time levels and multiply through by $2\dd$.
This results in
\begin{equation}
\label{eq:b10}
\begin{split}
&\left|\left|\rho^\half\rc^{n+1}\right|\right|^2_{L^2(\Omega)}-
\left|\left|\rho^\half\rc^0\right|\right|^2_{L^2(\Omega)}+
\left|\left|\la^{-\half}\phi^{n+1}\right|\right|^2_{L^2(\Omega)}-
\left|\left|\la^{-\half}\phi^{0}\right|\right|^2_{L^2(\Omega)} \\
&\qquad
+\dd^2\left(\theta-\frac{1}{4}\right)\left(
\left|\left|\la^{-\half}\rx^{n+1}\right|\right|^2_{L^2(\Omega)}-
\left|\left|\la^{-\half}\rx^{0}\right|\right|^2_{L^2(\Omega)} \right)\\
&\leq2\dd^2\left(\theta-\frac{1}{4}\right)\sum_{i=0}^n\left|\left|\la^{-\half}\zeta^{i+\half}
\right|\right|_{L^2(\Omega)}
\left( \left|\left|\la^{-\half}\rx^{i+1} \right|\right|_{L^2(\Omega)}+
\left|\left|\la^{-\half}\rx^{i} \right|\right|_{L^2(\Omega)}\right)\\
&\qquad+2\dd \sum_{i=0}^n\left|\left|\rc^{i+\half} \right|\right|_{L^2(\Omega)}
\left(\left|\left|\rho \pc \re^{i+\half} \right|\right|_{L^2(\Omega)}+
\left|\left|\rR^i\right|\right|_{L^2(\Omega)}\right)\\
&\qquad+2\dd \sum_{i=0}^n\left|\left|\lambda^{-\half} \zeta^{i+\half} 
\right|\right|_{L^2(\Omega)}
\left|\left|\lambda^{-\half} \phi^{i+\half}\right|\right|_{L^2(\Omega)}.
\end{split}
\end{equation}
Since $\left|\left|\la^{-\half}\xi^{i} \right|\right|_{L^2(\Omega)}
\leq \left|\left|\la^{-\half}\rx\right|\right|_{l^\infty(L^2)}$
and  $\left|\left|\rho^\half\rc^{i+\half} \right|\right|_{L^2(\Omega)}
\leq \left|\left|\rho^\half\rc\right|\right|_{l^\infty(L^2)}$, then
%
\begin{equation}
\label{eq:b11}
\begin{split}
&\left|\left|\rho^\half\rc^{n+1}\right|\right|^2_{L^2(\Omega)}+
\left|\left|\la^{-\half}\phi^{n+1}\right|\right|^2_{L^2(\Omega)}
+\dd^2\left(\theta-\frac{1}{4}\right)
\left|\left|\la^{-\half}\rx^{n+1}\right|\right|^2_{L^2(\Omega)}\\
&\leq4\dd^2\left(\theta-\frac{1}{4}\right)
\left|\left|\la^{-\half}\rx\right|\right|_{l^\infty(L^2)}
\left(\sum_{i=0}^n\left|\left|\la^{-\half}\zeta^{i+\half} 
\right|\right|_{L^2(\Omega)}\right)
\\
&\qquad+\frac{2\dd}{\rho_0^\half}\left|\left|\rho^\half \rc\right|\right|_{l^\infty(L^2)}
\left( \sum_{i=0}^n\left|\left|\rho \pc \re^{i+\half} \right|\right|_{L^2(\Omega)}+
\sum_{i=0}^n\left|\left|\rR^i\right|\right|_{L^2(\Omega)}\right)\\
&\qquad+2\dd\left|\left|\lambda^{-\half}\phi\right|\right|_{l^\infty(L^2)}
\left(\sum_{i=0}^n\left|\left|\lambda^{-\half} \zeta^{i+\half} \right|
\right|_{L^2(\Omega)}\right).
\end{split}
\end{equation}
%
Applying the algebraic inequality: 
$ab\leq \frac{\epsilon}{2} a^2+\frac{1}{2\epsilon} b^2$ to the right-hand side
of (\ref{eq:b11}) shows that
\begin{equation}
\label{eq:b12}
\begin{split}
&\left|\left|\rho^\half\rc^{n+1}\right|\right|^2_{L^2(\Omega)}+
\left|\left|\la^{-\half}\phi^{n+1}\right|\right|^2_{L^2(\Omega)}
+\dd^2\left(\theta-\frac{1}{4}\right)
\left|\left|\la^{-\half}\rx^{n+1}\right|\right|^2_{L^2(\Omega)}\\
&\leq\frac{1}{2}\dd^2\left(\theta-\frac{1}{4}\right)
\left|\left|\la^{-\half}\rx\right|\right|_{l^\infty(L^2)}^2
+8\left(\theta-\frac{1}{4}\right)
\left(\dd\sum_{i=0}^{N-1}\left|\left|\la^{-\half}\zeta^{i+\half}
\right|\right|_{L^2(\Omega)}\right)^2
\\
&\qquad+\frac{1}{2}\left|\left|\rho^\half \rc\right|\right|^2_{l^\infty(L^2)}+
C\dd^2 \left(\sum_{i=0}^{N-1}\left|\left|\rho \pc \re^{i+\half} \right|\right|_{L^2(\Omega)}+
\sum_{i=0}^{N-1}\left|\left|\rR^i\right|\right|_{L^2(\Omega)}\right)^2\\
&\qquad+\frac{1}{2}\left|\left|\lambda^{-\half}\phi\right|\right|^2_{l^\infty(L^2)}
+4\left(\dd \sum_{i=0}^{N-1}\left|\left|\lambda^{-\half} \zeta^{i+\half} \right|
\right|_{L^2(\Omega)}\right)^2.
\end{split}
\end{equation}
If we take the supremum over $n$ on the left-hand side and use the fact that 
$N\dd=T$, we conclude that
\begin{equation}
\label{eq:b13}
\begin{split}
&\left|\left|\rho^\half\rc\right|\right|^2_{l^\infty(L^2)}+
\left|\left|\la^{-\half}\phi\right|\right|^2_{l^\infty(L^2)}
\leq C\left|\left|\lambda^{-\half}\zeta\right|\right|^2_{l^\infty(L^2)}\\
&\qquad+
C\dd^2\left( \sum_{i=0}^{N-1}\left|\left|\rho \pc \re^{i+\half}
\right|\right|_{L^2(\Omega)}\right)^2+C\dd^2\left(\sum_{i=0}^{N-1}
\left|\left|\rR^i\right|\right|_{L^2(\Omega)}\right)^2.
\end{split}
\end{equation}
For the case  $\theta<\frac{1}{4}$, we can follow the analysis presented in
\cite{Kar-FE-Theta, Karaa-2012} to derive  error estimates similar to (\ref{eq:b13})
under  condition (\ref{eq:cfl2}).

To complete the proof, we need to bound each  term on the right-hand side 
of (\ref{eq:b13}).
The first term  can be bounded using
the approximation properties. Similarly, we have
$$
\dd\sum_{i=0}^{N-1}\left|\left|\rho\re^{i+\half}\right|\right|_{L^2(\Omega)}
\leq C\left(h^k||\ru||_{L^\infty(H^k(\Omega))}+\dd^2
\left|\left|\frac{\partial^3\ru}{\partial
t^3}\right|\right|_{L^1(0,T;L^2(\Omega))}\right).
$$
For the last term on  the right-hand side of (\ref{eq:b13}), we have
\begin{eqnarray*}
\dd\sum_{i=0}^{N-1}||\rR^i||_{L^2(\Omega)}
&\leq& C||\rR||_{l^\infty(L^2)}\\
&\leq& C\dd\sum_{i=1}^{N-1}||{\pmb r}^i||_{L^2(\Omega)}+C
||\rho(\Pi_h\rv^0-\rv^0)||_{L^2(\Omega)}\\
&&
+C\theta\dd^2\left|\left|\rho\pc\ru_{tt}^\half \right|\right|_{L^2(\Omega)}
+C\left|\left|\frac{1}{2\dd}\int_0^\dd \rho(\dd-t)^2\frac{\partial^3\ru}{\partial
t^3}(t)\,dt\right|\right|_{L^2(\Omega)}.
\end{eqnarray*}
To estimate $||{\pmb r}^i||_{L^2(\Omega)}$,
we make use of the identity
\begin{equation}\label{eq:t11}
\ps \ru^n=u_{tt}^n+\frac{1}{6\dd^2}\int_{-\dd}^{\dd}(\dd-|s|)^3
\frac{\partial^4\ru}{\partial t^4}(t^n+s)ds.
\end{equation}
From the Taylor's expansions of $\ru_{tt}^{n+1}$ and $\ru_{tt}^{n-1}$ about
$\ru_{tt}^{n}$;
$$
\ru^{n+1}_{tt}=\ru^n_{tt}+\dd
\ru_{ttt}^n+\int_{0}^{\dd}(\dd-|s|)\frac{\partial^4\ru}{\partial t^4}(t^n+s)\,ds,
$$
and
$$
\ru^{n-1}_{tt}=\ru^n_{tt}-\dd
\ru_{ttt}^n+\int_{-\dd}^{0}(\dd-|s|)\frac{\partial^4\ru}{\partial t^4}(t^n+s)\,ds,
$$
we obtain
\begin{equation}\label{eq:t12}
\ru^{n;\theta}_{tt}=\ru_{tt}^n+\theta\int_{-\dd}^{\dd}(\dd-|s|)
\frac{\partial^4\ru}{\partial t^4}(t^n+s)\,ds.
\end{equation}
Subtracting  (\ref{eq:t11}) from (\ref{eq:t12}) yields
$$
\ru^{n;\theta}_{tt}-\ps \ru^n=
\frac{1}{6\dd^2}\int_{-\dd}^{\dd}(|s|-\dd)^3\frac{\partial^4\ru}{\partial t^4}(t^n+s)\,ds
-\theta\int_{-\dd}^{\dd}(|s|-\dd)\frac{\partial^4\ru}{\partial t^4}(t^n+s)\,ds.$$
Hence, 
$$
||{\pmb r}^i||_{L^2(\Omega)}^2=
||\rho\left(\ru^{n;\theta}_{tt}-\ps \ru^n\right)||_{L^2}
\leq C\dd^3
\int_{-\dd}^{\dd}
\left|\left|\rho^\half\frac{\partial^4\ru}{\partial t^4}
(t^n+s)\right|\right|_{L^2(\Omega)}^2\,ds\leq C\dd^4\left|\left|\rho^\half
\frac{\partial^4\ru}{\partial t^4}
\right|\right|_{L^\infty(L^2)}^2,$$
and therefore
$$
\dd\sum_{i=1}^{n}||{\pmb r}^i||_{L^2(\Omega)}\leq
C\dd^2\left|\left|\rho^\half\frac{\partial^4\ru}{\partial
t^4}\right|\right|_{L^\infty(L^2)}\sum_{i=1}^{n}\dd
\leq CT\dd^2\left|\left|^2\rho^\half\frac{\partial^4\ru}{\partial t^4}
\right|\right|_{L^\infty(L^2)}.
$$
Similarly, we have
$$
\left|\left|\rho\pc\ru_{tt}^\half \right|\right|_{L^2(\Omega)}^2=
\left|\left|\dd\int_0^\dd \rho\frac{\partial^3\ru}{\partial
t^3}(t)\,dt\right|\right|_{L^2(\Omega)}^2
\leq C\dd^3\int_0^\dd
\left|\left|\rho^\half\frac{\partial^3\ru}{\partial
t^3}\right|\right|_{L^2(\Omega)}^2dt
\leq C\dd^4
\left|\left|\rho^\half\frac{\partial^3\ru}{\partial
t^3}\right|\right|_{L^\infty(L^2)}^2,
$$
and 
$$
\left|\left|\frac{1}{2\dd}\int_0^\dd \rho(\dd-t)^2\frac{\partial^3\ru}{\partial
t^3}(t)\,dt\right|\right|_{L^2(\Omega)}^2\leq C\dd^3\int_0^\dd
\left|\left|\rho^\half\frac{\partial^3\ru}{\partial
t^3}\right|\right|_{L^2(\Omega)}^2dt
\leq C\dd^4
\left|\left|\rho^\half\frac{\partial^3\ru}{\partial
t^3}\right|\right|_{L^\infty(L^2)}^2.
$$
Finally, using the approximation property (\ref{eq:w9}) and combining all the
bounds, we arrive at
$$
\dd\sum_{i=0}^{N-1}||\rR^i||_{L^2(\Omega)}\leq C(h^k+\dd^2),
$$
which completes the proof of the desired estimate.
\end{proof}
%

\noindent{\bf Remarks.}  It is worthwhile to mention that the time discretization method is fourth-order accurate
when $\theta=1/12$. To preserve the temporal accuracy of the finite element scheme one 
has to modify (\ref{eq:www3}) carefully to obtain an appropriate initial value 
$\rU^1$. The analysis presented in \cite{Kar-FE-Theta} can be used to
derive  optimal a priori error estimates in this case.


\section{Conclusions}


We proposed and analyzed a family of fully discrete mixed finite element 
schemes for solving the acoustic wave equation. 
We derived stability conditions for conditionally implicit
stable schemes covering the explicit
case treated by Jenkins,  Rivi\`ere and  Wheeler \cite{JRW}.  
The error estimates established provided optimal convergence rates for 
the use of mixed finite elements methods in solving the acoustic wave equation.

\end{document}